\documentclass[11pt]{article}

\usepackage{latexsym,color,amsmath,amsthm,amssymb,amscd,amsfonts,mathtools}
\usepackage{graphicx}
\usepackage{enumitem}
\usepackage{cancel} 
\usepackage{chngcntr}
\usepackage{apptools}
\usepackage{titlesec}

\titleformat{\paragraph}
  {\normalfont\normalsize\bfseries}
  {\theparagraph}
  {1em}
  {}

\titlespacing*{\paragraph}
  {0pt}{1.5ex plus .2ex}{0.8ex}

\newtheorem{conj}{Conjecture}[section]
\newtheorem{theorem}[conj]{Theorem}
\newtheorem{lemma}[conj]{Lemma}

\newtheorem{coro}[conj]{Corollary}
\newtheorem{remark}[conj]{Remark}

\newcommand\independent{\protect\mathpalette{\protect\independent}{\perp}} 
\def\independent#1#2{\mathrel{\rlap{$#1#2$}\mkern2mu{#1#2}}}

\newcommand{\X}{\mathcal{X}}

\newcommand{\eps}{\varepsilon}

\date{\vspace{-5ex}}

\author{Heshan Aravinda}

\title{On Two Conjectures Related to the Boros-Moll Sequences}

\begin{document}
\maketitle

\begin{abstract}\noindent
The Boros-Moll sequences $\{d_i(m)\}_{0\leq i \leq m}$ are defined as $$d_i(m) = 2^{-2m}\sum_{k=i}^m 2^k \binom{2m-2k}{m-k}\,\binom{m+k}{k}\,\binom{k}{i}.$$ Consider the ratio sequence $$u_i(m) = \dfrac{d_{i-1}(m)\, d_{i+1}(m)}{d_i^2(m)} .$$
In \cite{CG}, Chen and Gu conjectured that $\{u_i(m)\}_{2\leq i \leq m-2}$ is both reverse ultra log-concave and log-concave. In this paper, we prove the reverse ultra log-concavity conjecture using bounds of Chen–Gu and Zhao, and prove the log-concavity conjecture asymptotically by showing that $\{u_i(m)\}_{2\leq i \leq m-2}$ is strictly log-concave for all sufficiently large $m$. The key ingredient in the latter result is a recurrence of Kauers and Paule, which we interpret as a nonlinear discrete dynamical system through a backward map. We construct an approximation to the ratio sequence using its stable limiting fixed point and combine localization and contraction arguments with finite-difference estimates and separate interior and edge analyses to obtain the desired strict log-concavity.   

\end{abstract}


\section{Introduction}

The Boros-Moll sequences were introduced by Boros and Moll (see \cite{BM,BMI,BMII,M}) in their study of the quartic integral
$$\int_0^\infty \dfrac{1}{(t^4+ 2xt^2+1)^{m+1}}\,dt\,,\,\,\,\,\,\,x>-1 \,\,\,\text{and $m \in \mathbb{N}$}.$$ They showed that
$$\int_0^\infty \dfrac{1}{(t^4+ 2xt^2+1)^{m+1}}\,dt = \dfrac{\pi}{2^{m+ 3/2}\,(x+1)^{m+1/2}}\,P_m(x)\,,$$
where $$P_m(x) = 2^{-2m} \sum_{k=0}^m 2^k \binom{2m-2k}{m-k}\,\binom{m+k}{k}\,(x+1)^k.$$
The polynomials $P_m(x)$, called the {\it Boros-Moll polynomials}, belong to a special class of Jacobi polynomials. Writing 
$$P_m(x) = \sum_{i=0}^m d_i(m)x^i\,,$$
the sequence $\left\{d_i(m)\right\}_{0\leq i\leq m}$ of the coefficients is called the {\it Boros-Moll sequence}, and they have the following explicit form
$$d_i(m) = 2^{-2m}\sum_{k=i}^m 2^k \binom{2m-2k}{m-k}\,\binom{m+k}{k}\,\binom{k}{i}.$$
These sequences have been studied extensively because of their strong shape properties. For example, Boros and Moll proved that $\left\{d_i(m)\right\}_{0\leq i\leq m}$ is unimodal (see \cite{BMIII}) and conjectured that it is log-concave \cite{M}. This conjecture was subsequently proved by Kauers and Paule  \cite{KP} using a computer algebra approach. In \cite{CX_ratio}, Chen and Xia established strict ratio monotonocity. Higher-order refinements have also been obtained, including \(2\)- and \(3\)-log-concavity \cite{CDY} and higher-order Turán inequalities \cite{G,Z}. \vskip2mm \noindent
Recall that a sequence of real numbers $\{a_k\}_{0\leq k\leq n}$ is said to be {\it log-concave} if $a_k^2 \geq a_{k-1}\,a_{k+1}$ for $1 \leq k \leq n-1$. A stronger property is ultra log concavity; the sequence $\{a_k\}_{0 \leq k \leq n}$ is {\it ultra log-concave} if the normalized sequence $\left\{a_k/\binom{n}{k}\right\}_{0\leq k\leq n}$ is log-concave (see \cite{L}). Equivalently, 
\begin{equation}\label{defn_ULC}
k(n-k)a_k^2 - (n-k+1)(k+1)a_{k-1}\,a_{k+1} \geq 0.
\end{equation} 
A sequence is said to be {\it reverse ultra log-concave} if it satisfies the reverse inequality of \eqref{defn_ULC}, i.e.
\begin{equation}\label{defn_reverse_ULC}
k(n-k)a_k^2 - (n-k+1)(k+1)a_{k-1}\,a_{k+1} \leq 0.
\end{equation}
A classic example of a sequence that is both log-concave and reverse ultra log-concave is the sequence of coefficients of Bessel polynomials (see \cite{HS}).\vskip2mm 
The motivation for our work stems from the work of Chen and Gu \cite{CG} who proved that the Boros–Moll sequences are reverse ultra log-concave. They further showed that, asymptotically, these sequences lie on the borderline between ultra-log-concavity and reverse ultra log-concavity. Motivated by their analysis, they proposed the following two conjectures (see \cite[Conjecture 4.3, Conjecture 4.4]{CG})) on the sequence $\{u_i(m)\}$, where
$$u_i(m) =\dfrac{d_{i-1}(m)\, d_{i+1}(m)}{d_i^2(m)}.$$
\begin{conj}\label{reverse_ULC_conj}
   For $m \geq 2$, the sequence $\{u_i(m)\}_{2 \leq i \leq m-2}$ is reverse ultra log-concave.
\end{conj}
\begin{conj}\label{LC_BM_conj}
    For $m \geq 2$, the sequence $\{u_i(m)\}_{2 \leq i \leq m-2}$ is log-concave.
\end{conj} 
Our first main result settles Conjecture \ref{reverse_ULC_conj}.
\begin{theorem}\label{reverse_ULC}
    For $m \geq 2$, the sequence $\{u_i(m)\}_{2 \leq i \leq m-2}$ is reverse ultra log-concave.
\end{theorem} \noindent
The proof of Theorem \ref{reverse_ULC} follows directly from sharp bounds for $u_i(m)$ obtained by Chen–Gu and Zhao; see Section 2 for relevant details.
\vskip2mm
Our second main result establishes Conjecture \ref{LC_BM_conj} for all sufficiently large $m$. 
\begin{theorem} \label{eventual_LC}
    There exists $M \in \mathbb{N}$ such that for all $m \geq M$, the sequence $\{u_i(m)\}_{2 \leq i \leq m-2}$ is strictly log-concave.
\end{theorem}
The proof of Theorem \ref{eventual_LC} requires a more detailed asymptotic analysis. We introduce the coefficient ratios
$$ r_i(m)=\frac{d_i(m)}{d_{i-1}(m)} $$
and utilize its nonlinear backward recurrence as a discrete dynamical system. In the regime where $\theta=i/m$ stays bounded away from 0 and 1, which we call the {\it interior regime}, the dynamics possesses a stable limiting fixed point function $$R_0(\theta) = \dfrac{1- \theta}{\theta}$$ 
which we use to construct an approximation to $r_i$. We combine this approximation with localization and contraction arguments to obtain sufficiently precise control of the error and its third finite difference. The interior estimates are, however, not uniform as $\theta$ approaches the endpoints, requiring separate analyses of the {\it left-edge regime} $\theta \to 0$ and the {\it right-edge regime} $\theta \to 1$. At the left edge, a further transition occurs at the scale $i \asymp \sqrt{m}$, determined by the parameter $h/\theta^2=m/i^2$, leading to separate {\it inner} and {\it outer left-edge} arguments.
\vskip2mm  
The remainder of the paper is organized as follows. In Section 2, we prove Theorem \ref{reverse_ULC}. Section 3 is devoted to the proof of Theorem \ref{eventual_LC}, treating separately the interior, left-edge, and right-edge regimes. The auxiliary localization and finite-difference estimates used in the proof are included in the Appendix.

\section{Reverse  ultra log-concavity}
In this section, we prove Theorem \ref{reverse_ULC} using the bounds of Chen–Gu and Zhao recalled below.
   \begin{proof}
Since the sequence $\{u_i(m)\}$ starts at $i=2$, by letting $a_k = u_{k+2}$ with $k = i-2$ and $n = m-4$ in \eqref{defn_reverse_ULC}, we wish to show
       \begin{equation}\label{revULC}
           u_i^2(m) \leq \dfrac{(m-i-1)(i-1)}{(m-i-2)(i-2)}\,u_{i-1}(m)\,u_{i+1}(m) \,\,\,\text{for $3 \leq i \leq m-3$.}
           \end{equation}
       To this end, we use the following bounds on $u_i(m)$ (see \cite[Theorem 1.1]{CG}, \cite[Theorem 3.1]{Z}). 
       \begin{equation} \label{CGZ_bound}
       f_i(m) < u_i(m) < g_i(m)\,,\,\,\,\,\,\,\, \text{for $m\geq 2$ and $1 \leq i \leq m-1$}\,,\end{equation}
       where $f_i(m) = \dfrac{(m-i)i}{(m - i+1)(i+1)}$  and $g_i(m) = \dfrac{(m-i)\,i\,(m+i^2 +1)}{(m-i+1) (i+1)(m+i^2)}.$\vskip2mm \noindent
   Therefore, to prove \eqref{revULC}, it is enough to prove
   \begin{equation} \label{equ_revulc}
     \dfrac{(m-i-1)(i-1)}{(m-i-2)(i-2)}\,f_{i-1}(m)\, f_{i+1}(m) - g_i^2(m) \geq 0
   \end{equation}
  After plugging in for $f_{i-1}, f_{i+1}$ and $g_i$, and clearing the denominator, \eqref{equ_revulc} becomes     
   \begin{equation} \label{reverse_ulc_ineq}
        ((m-i)^2 -1)^2(i^2-1)^2(m-i+1)(i+1)(m+i^2)^2 - i^3\,((m-i)^2 - 4)(i^2 - 4)(m-i)^3\,(m+i^2 +1)^2  \geq 0
   \end{equation}
   Let $i = x+3$ and $m -i = y+3 $, where $x,y \geq 0$. With $x$ and $y$, the left-hand side of \eqref{reverse_ulc_ineq} can be simplified further to obtain a two-variable polynomial of the form $$\sum_{k=0}^9 P_k(y)\,x^k\,,$$
where \vskip4mm \noindent
$P_0(y) = 121y^7+5431y^6+95110y^5+845914y^4+4186273y^3+11628271y^2+16923360y+10080000 $ \vskip2mm \noindent
$P_1(y) = 151y^7+8343y^6+169020y^5+1668090y^4+8938203y^3+26435919y^2+40500050y+25220640 $ \vskip2mm \noindent
$P_2(y) = 70y^7+5342y^6+131337y^5+1466043y^4+8570004y^3+27030312y^2+43516157y+28227679$ \vskip2mm \noindent
$P_3(y) = 14y^7+1840y^6+58643y^5+755355y^4+4846924y^3+16306730y^2+27521947y+18529731 $ \vskip2mm \noindent
$P_4(y) = y^7+367y^6+16631y^5+252121y^4+1783711y^3+6395193y^2+11284673y+7859479$ \vskip2mm \noindent
$P_5(y) = 41y^6+3102y^5+56565y^4+442861y^3+1689741y^2+3108971y+2233459 $ \vskip2mm \noindent
$P_6(y) = 2y^6+377y^5+8509y^4+74066y^3+300434y^2+575111y+425173) $ \vskip2mm \noindent
$P_7(y) = 28y^5+821y^4+8018y^3+34596y^2+68817y+52273$ \vskip2mm \noindent
$P_8(y) = y^5+45y^4+506y^3+2334y^2+4827y+3765 $ \vskip2mm \noindent
$P_9(y) = y^4+14y^3+70y^2+151y+121 $
\vskip3mm
We conclude since every coefficient of the polynomial in $x$ is a polynomial in $y$ with non-negative coefficients.
 \end{proof}
\newpage
\section{Eventual strict log-concavity}
We now turn to the proof of Theorem \ref{eventual_LC}. The key step is to reformulate the log-concavity of $\{u_i(m)\}$ in terms of the consecutive coefficient ratios.
Let $$r_i = r_i(m) \coloneq \dfrac{d_i(m)}{d_{i-1}(m)}.$$Then $$u_i = \dfrac{r_{i+1}}{r_i}.$$ 
Define
$$D_{m,i} \coloneq 2\log u_i - \log u_{i-1} - \log u_{i+1}.$$
Thus proving Theorem \ref{eventual_LC}  is equivalent to proving $D_{m,i}>0$. Moreover,
$$D_{m,i} = \log r_{i-1} - 3 \log r_i + 3 \log r_{i+1} - \log r_{i+2} = -\Delta^3 (\log r_{i-1}).$$
So $$D_{m,i}>0 \iff -\Delta^3 (\log r_{i-1})>0 .$$
We will establish this inequality asymptotically by analyzing the ratios $r_i$ through their nonlinear backward recurrence. Let us begin with the following recurrence relation for Boros-Moll sequence $\{d_i(m)\}_{0\leq i \leq m}$ (see \cite{KP}).
\begin{equation}\label{BM_recurrence}
    (m+2-i)(m+i-1)\,d_{i-2}(m) - (i-1)(2m+1)d_{i-1}(m) + i(i-1)\,d_i(m) = 0\,,
\end{equation}
which is valid for $m\geq 1$.\vskip2mm \noindent
For convenience, we shift the index $i \mapsto i+1$ and divide by $r_i$ \vskip2mm
\begin{equation}\label{shifted_recurrence}
   (m+1-i) (m+i) - i (2m+1)\,r_i + i(i+1)\,r_i\,r_{i+1} = 0.
\end{equation}
Solving for $r_i$
$$r_i = \dfrac{(m-i+1)(m+i)}{i\, (2m+1 -(i+1))r_{i+1}}.$$
Define the map \begin{equation}\label{map_Phi}
\Phi_{m,i}(x) = \dfrac{(m-i+1)(m+i)}{i (i+1)}\left(\dfrac{2m+1}{i+1} - x\right)^{-1},
\end{equation}
so that $r_i = \Phi_{m,i}(r_{i+1})$.\vskip2mm
Thus, \eqref{shifted_recurrence} can be viewed as a discrete dynamical system for the ratios $r_i$, where the map $\Phi_{m,i}$ acts as a backward map describing the evolution from $r_{i+1}$ to $r_i$. This viewpoint is key to our analysis below: identifying the stable limiting function of this dynamics provides the contraction mechanism that allows us to control the error between the exact ratios and their asymptotic approximation.\vskip2mm
\subsection{Interior}
We first treat the interior regime $i \asymp m$. \vskip2mm
Fix $0< \epsilon < 1/2$ and restrict to $\epsilon m \leq i \leq (1- \epsilon)m$.
Let $\theta_i = \dfrac{i}{m}$ and $h = \dfrac{1}{m}$
so that $\theta_i \in [\epsilon , 1-\epsilon]$ and $\theta_{i+1}  = \theta_i + h$.\vskip2mm \noindent
Rewriting \eqref{map_Phi} in terms of $\theta$ and $h$:
\begin{equation}\label{theta_h_map}
    \Phi_{\theta , h}(x) = \dfrac{(1-\theta+h)(\theta +1)}{\theta (\theta+h)}\left(\dfrac{2+h}{\theta+h} - x\right)^{-1}.
\end{equation}

\paragraph*{Stable fixed point}

Note that the introduction of the parameters $\theta_i$ and $h$ is natural for the asymptotic analysis in the interior regime: the variable $\theta_i$ tracks the asymptotic location of the index $i$ relative to $m$ while $h$ is the small parameter governing the limit $m \to \infty$. In particular, $\theta_{i+1}  = \theta_i + h$ so consecutive indices correspond to $ O(h)$ change in the variable $\theta$.\vskip2mm \noindent
As observed, the variable $\theta$ changes only by $O(h)$ at each step in the recurrence, forcing the map $ \Phi_{\theta , h}$ to vary slowly in the interior. Therefore it is natural to study the limiting map 
\begin{equation}\label{lim_map}
     \Phi_{\theta,0}(x) = \dfrac{1-\theta^2}{\theta^2} \left(\dfrac{2}{\theta} - x\right)^{-1},
\end{equation}
 existence of its fixed point and their stability.
In order to investigate the fixed point, say $R(\theta)$, consider the equation
\begin{equation}\label{fixed_eqn}
    R(\theta) = \Phi_{\theta,0}(R(\theta)),
\end{equation}
which after simplifying, is equivalent to
$$R(\theta)^2 \theta^2 - 2 \theta \,R(\theta)+ 1- \theta^2 = 0.$$
The corresponding roots are
$$R_-(\theta) = \dfrac{1 - \theta}{\theta}\,\,\,\,\text{and}\,\,\,\,R_+(\theta) = \dfrac{1 + \theta}{\theta}.$$
However, only $R_-(\theta)$ represents the stable fixed point since 
$$\left. \frac{\partial \Phi_{\theta,0}}{\partial x} \right |_{R_1(\theta)} = \dfrac{1- \theta}{1+ \theta}<1 \,\,\,\,\text{and}\,\,\,\,\left. \frac{\partial \Phi_{\theta,0}}{\partial x} \right |_{R_2(\theta)} = \dfrac{1+ \theta}{1- \theta}>1.$$
Let $R_0(\theta) = \dfrac{1-\theta}{\theta}$. The stability of $R_0(\theta)$ suggests that the ratio $r_i$ should track this limiting function as $m \to \infty$. However, since $\theta_i$ varies by $O(h)$ at each step, $R_0(\theta_i)$ alone does not exactly follow the recurrence. We therefore construct a first-order approximation of the form
\begin{equation*}
    \hat{r}_i = R_0(\theta_i) + hR_1(\theta_i),
\end{equation*}
where $R_1(\theta_i)$ is chosen so that its residual is of order $h^2$. The contractivity near the stable fixed point will then allow us to convert this residual estimate into a quantitative control of the error $r_i- \hat{r}_i$. We carry out the details in the following subsections.

\paragraph*{First-order approximation}
We seek a first-order approximation of the form
\begin{equation} \label{approx}
    \hat{r}_i = R_0(\theta_i) + hR_1(\theta_i),
\end{equation}
where $R_1(\theta_i)$ is chosen to be smooth and its residual is of order $h^2$.
    Since $\theta_{i+1} = \theta_i + h$,
    \begin{equation*}
        \hat{r}_{i+1} = R_0(\theta_i + h) + hR_1(\theta_i +h) = R_0(\theta_i) + h(R_0^\prime(\theta_i)+ R_1(\theta_i)) + O_\epsilon(h^2).
    \end{equation*}
    Here, the second equality is due to Taylor expansions of $R_0$ and $R_1$ about $\theta$.\vskip2mm \noindent
    For fixed $\theta$, the map $(h, x) \mapsto \Phi_{\theta, h}(x)$ is smooth in $h$ and $x$. As $h \to 0$, $(h, \hat{r}_{i+1} ) \to (0, R_0(\theta_i) )$, so expanding $\Phi_{\theta, h}$ about $(0, R_0(\theta))$ gives
    $$ \Phi_{\theta, h}(R_0+ h(R_0^\prime + R_1) + O(h^2)) =\Phi_{\theta, 0}(R_0) + h \partial_h\Phi_{\theta, 0}(R_0) + h (R_0^\prime + R_1) \partial_x \Phi_{\theta, 0}(R_0) + O\epsilon(h^2).$$
    Since  $R_0(\theta) = \Phi_{\theta,0}(R_0(\theta))$, choose $R_1$ so that the coefficients of $h$ in $\Phi_{\theta, h}(\hat{r}_{i+1})$ and $\hat{r}_i$ agree.\vskip2mm \noindent
    This yields
    \begin{equation}\label{R_1_eqn}
        R_1 = \partial_h\Phi_{\theta, 0}(R_0)+ (R_0^\prime + R_1) \partial_x \Phi_{\theta, 0}(R_0) \implies R_1 = \dfrac{\partial_h \Phi_{\theta, 0}(R_0) + R_0^\prime \partial_x \Phi_{\theta, 0}(R_0)}{1 - \partial_x\Phi_{\theta, 0}(R_0)}  \end{equation}
    It remains to compute the derivatives. Note that $R_0^\prime(\theta) = - \dfrac{1}{\theta^2}$ and direct computation gives
    $$\partial_h \Phi_{\theta, 0}(R_0) = \dfrac{3\theta^2 -2 \theta+1}{\theta^2 (1+ \theta)}\,\,\,\,\,\text{and}\,\,\,\,\partial_x \Phi_{\theta, 0}(R_0) = \dfrac{1-\theta}{1+\theta}$$
    Plugging these into \eqref{R_1_eqn} results in
    $$R_1(\theta) = \dfrac{3 \theta - 1}{2\theta^2}.$$
    
\paragraph*{Calculation of residual}
    Define the one-step residual
    \begin{equation} \label{residual}
        \tau_i \coloneq \Phi_{\theta_i, h}(\hat{r}_{i+1}) - \hat{r}_i
    \end{equation}
    Then $\tau_i = \left.\tau(\theta, h) \right|_{\theta = \theta_i}$ where $\tau(\theta, h)$ is the smooth residual defined as
    \begin{equation} \label{smooth_residual}
        \tau(\theta,h) = \Phi_{\theta, h}(R_0(\theta + h) + hR_1(\theta +h)) - R_0(\theta)-hR_1(\theta).
    \end{equation}
    Clearly $ \tau(\theta,0) = 0$. Moreover, the derivative of $\tau$ w.r.t $h$ evaluated at $h=0$ is
    $$\partial_h \tau(\theta,0) = \partial_h \Phi_{\theta, 0}(R_0) +  (R_0^\prime + R_1) \partial_x \Phi_{\theta, 0}(R_0) -R_1 = 0.$$
    Here the second equality follows from \eqref{R_1_eqn}, the condition needed for $R_1$.\vskip2mm \noindent
    In what follows, we work on the slightly larger interval $I_\epsilon = [\epsilon/2 , 1- \epsilon/2]$ to accommodate finite shifts in $\theta$ arising from the recurrence and subsequent difference estimates so that for sufficiently small $h$, all shifted arguments needed for indices with $\theta_i \in [\epsilon , 1- \epsilon]$  remain in $I_\epsilon$.\vskip2mm \noindent
    Notice that denominators arise from \eqref{smooth_residual} remain uniformly away from 0 for sufficiently small $h$. Indeed, the three denominators show up in the expression are
    $$\theta, \theta+h,\,\,\, \text{and}\,\,\,\, \dfrac{2+h}{\theta+h} -R_0(\theta + h) - hR_1(\theta +h). $$
    Clearly, $\theta , \theta+h >0$ and it is not too difficult to verify $\dfrac{2+h}{\theta+h} -R_0(\theta + h) - hR_1(\theta +h) \geq c_{\epsilon}>0$ for all $\theta \in I_\epsilon$ and sufficiently small $h$.\vskip2mm \noindent
   As a consequence, the function $\tau$ is smooth on $I_\epsilon \times \{0\}$. In particular, all of its derivatives exist, are continuous, and therefore are bounded on the compact strip.\vskip2mm \noindent
   Fix $\theta$ and let $f(h) = \tau(\theta, h)$ which is a smooth function in $h$. Applying Taylor's theorem to $f(h)$
   \begin{align*}
       f(h) &= f(0)+ hf^\prime(0) + \int_0^h (h-s)f''(s)ds\\
       & = \int_0^h (h-s)f''(s)ds \,\,\,\,\,\,\,\,\,\,\,\,\text{(because $f(0) = f^\prime(0) = 0)$}
   \end{align*}
  Let us introduce the parameter $t$ such that $s = th$ where $0\leq t \leq 1$. Rewriting the above equation in terms of $t$
   \begin{align*}
       f(h) = h^2\, \int_0^1 (1-t)f''(th)\,dt.
          \end{align*}
Equivalently,
$$\tau(\theta , h) = h^2\, \int_0^1 (1-t)\partial_h^2 \tau(\theta, th)\,dt\,,$$
where $\partial_h^2 \tau$ denotes the second partial derivative of $\tau$ w.r.t $h$.\vskip2mm \noindent
Define
$$S(\theta, h) = \int_0^1 (1-t)\partial_h^2 \tau(\theta, th)\,dt$$
Then $\tau(\theta , h)  = h^2 S(\theta,h)$. Since $\tau$ is smooth, for fixed $k$, we have
\begin{align*} \partial_\theta^k S &= \int_0^1 (1-t)\partial_\theta^k \partial_h^2 \tau(\theta, th)\,dt\\
\implies
    |\partial_\theta^k S| & \leq |\partial_\theta ^k \partial_h^2 \tau(\theta, th)| \int_0^1 (1-t)\,dt\\
    & = \sup |\partial_\theta^k \partial_h^2\, \tau| \int_0^1 (1-t)\,dt\\
    & = \dfrac{1}{2} \sup |\partial_\theta^k \partial_h^2\, \tau|
\end{align*}
Therefore, $\displaystyle \sup_{\theta \in I_\epsilon} |\partial_\theta^k S| \leq C_{\epsilon, k}$ uniformly for sufficiently small $h$. This also implies that $S$ is smooth on the compact strip under consideration. \vskip2mm \noindent
Let $\Delta_h S(\theta,h) = S(\theta+h, h) - S(\theta, h)$. By the Fundamental Theorem of Calculus, we can write
$$\Delta_h S(\theta,h) = S(\theta+h, h) - S(\theta, h) = \int_0^h \partial_\theta S (\theta+t,h)\,dt.$$
Then the repeated application of FTC gives
\begin{align*}\Delta_h^k S(\theta,h) &= \int_{[0,h]^k} \partial_\theta^k S (\theta+t_1+t_2+\dots+t_k,h)\,dt_1\,dt_2\,\dots dt_k.
\end{align*}
This implies
\begin{align*}|\Delta_h^k S(\theta,h)| & \leq  |\partial_\theta^k S| \int_{[o,h]^k}1\, dt_1\,dt_2\dots dt_k\\
& \leq  h^k C_{\epsilon, k}
\end{align*}
Since $\tau_i = \tau(\theta_i,h) = h^2 S(\theta_i, h)$ and $\theta_{i+1} = \theta_i+h$, $\Delta \tau_i = h^2 \Delta_h S(\theta_i,h)$. Therefore,
$$|\Delta^k \tau_i| \leq C_{\eps,k}h^{2+k}.$$
Equivalently, 
\begin{equation}\label{difference_tau}
|\Delta^k \tau_i| = O_\epsilon(h^{2+k}).
\end{equation}
Thus, the residual is not merely $O(h^2)$; it varies smoothly with parameter $\theta$, with each successive finite difference gaining an additional factor of $h$. Therefore, this quantifies how closely the first-order approximation constructed earlier follows the exact recurrence relation.\vskip2mm
In the next subsection, we will establish that this small varying residual, together with contraction, implies that $\hat{r}_i$ is close to $r_i$. A key ingredient in the argument is uniform localization. (see Corollary \ref{Localization}(i))
\paragraph*{Localization and error estimate}
    Set $$E_i\coloneq  r_i - \hat{r}_i$$
    We know $r_i = \Phi_{\theta_i,h}(r_{i+1})$ and $\hat{r}_i = \Phi_{\theta_i, h}(\hat{r}_{i+1}) -\tau_i$. Then
    $E_i = \Phi_{\theta_i,h}(r_{i+1}) - \Phi_{\theta_i, h}(\hat{r}_{i+1}) + \tau_i $.
    \begin{equation}\label{absolute}
    \implies |E_i| \leq | \Phi_{\theta_i,h}(r_{i+1}) - \Phi_{\theta_i, h}(\hat{r}_{i+1})| + C_{\epsilon,k}h^2.
    \end{equation}
    By the uniform localization established in Appendix (see Corollary \ref{Localization}(i)), $r_i = R_0(\theta_i)+ O_\epsilon(h)$ uniformly on the interior. Since $R_1$ is uniformly bounded there, $\hat{r}_i = R_0(\theta_i)+ O_\epsilon(h)$ as well. Hence, for sufficiently small $h$, both $r_{i+1}$ and $\hat{r}_{i+1}$, and therefore, the line segment joining $r_{i+1}$ and $\hat{r}_{i+1}$ lies in a sufficiently small uniform neighborhood of $R_0(\theta)$. Hence, by the Mean Value Theorem, there exists $\zeta_i$ between $r_{i+1}$ and $\hat{r}_{i+1}$ such that
    $$| \Phi_{\theta_i,h}(r_{i+1}) - \Phi_{\theta_i, h}(\hat{r}_{i+1})|  = |\partial_x \Phi_{\theta_i, h}(\zeta_i)||r_{i+1} - \hat{r}_{i+1}|.$$
    Earlier, we showed that $\partial_x \Phi_{\theta,0}(R_0(\theta)) = \dfrac{1-\theta}{1+\theta}<1$. Since $\theta \in I_\epsilon$, set  $$q_\epsilon = \sup_{\theta \in I_\epsilon} |\partial_x \Phi_{\theta,0}(R_0(\theta))|\,$$ which is less than 1. Then by continuity and compactness, there exist a constant $q_\epsilon < \rho_\epsilon <1$ and a sufficiently small uniform neighborhood of the stable limit fixed-point function $R_0(\theta)$ such that, for all sufficiently small $h$,
$$ |\partial_x \Phi_{\theta,h}(x)|<\rho_\eps.$$
       Since $\zeta_i$ lies in the same uniform neighborhood by localization, 
    $$| \Phi_{\theta,h}(r_{i+1}) - \Phi_{\theta_i, h}(\hat{r}_{i+1})| \leq \rho_\epsilon|r_{i+1} - \hat{r}_{i+1}|.$$
    Combining this with \eqref{difference_tau}, we get
    $$|E_i| \leq \rho_\epsilon \,|E_{i+1}| + C_{\epsilon}h^2. $$
    Let us now choose a terminal point $J$ such that
    $$J = \lfloor (1- \epsilon/2)m \rfloor.$$
    If $\eps m \leq i \leq (1-\epsilon)m$, then $\theta_J \in I_\epsilon$ for all $i \leq j \leq J$. Write $L = J-i$, and by iteration
    \begin{align*}
        |E_i| &\leq \rho_\epsilon \,|E_{i+1}| + C_{\epsilon}h^2 \\
          &\leq \rho_\epsilon^2 |E_{i+2}| + C_\epsilon h^2 (1+ \rho_\epsilon)\\
          &\leq \dots\\
          & \leq \rho_\epsilon^L |E_J| + C_\epsilon h^2 \sum_{n=0}^{L-1} \rho_\epsilon^n
    \end{align*}
    Note that the localization holds uniformly for every index encountered during this iteration, hence the same contraction factor $\rho_\epsilon<1$ applies at every step.\vskip2mm
    It follows from the localization, together with the boundedness of $\hat{r}_j$ on the compact interval that $E_J = O_\epsilon(1)$. In particular, $E_j = O_\epsilon(1)$ for each $j$ in the enlarged interior interval. Moreover, since $i\leq (1-\epsilon)m$ and $\lfloor x\rfloor \geq x-1$ $$L = J-i \geq \lfloor(1-\epsilon/2)m\rfloor - (1-\epsilon)m \geq (1-\epsilon/2)m -1 - (1-\epsilon)m = \dfrac{\epsilon}{2}m - 1$$
    which is at least $k_\epsilon m$ for all sufficiently large $m$. Therefore,
    $$\rho_\epsilon^L = e^{-a_\epsilon/h}\,,$$
    where $a_\epsilon \leq  - k_\epsilon \log\rho_\epsilon $.\vskip2mm
    This implies $\rho_\epsilon^L$ decays exponentially as $h \to 0$, and therefore, $\rho_\epsilon^L = o(h^2)$ (in fact, it is $o(h^N)$ for $N>0$).\vskip2mm \noindent
    Revisiting the iteration for $|E_i|$,
    \begin{equation}\label{E_i}
    |E_i| = O_\epsilon(h^2)\,\,\,\,\,\text{ uniformly for $\epsilon m \leq i \leq (1- \epsilon)m$. }
    \end{equation}
   Thus the first-order approximation introduced above is valid uniformly in the interior:
    \begin{equation}\label{first_order}
    r_i = R_0(\theta_i) + hR_1(\theta_i) + O_\epsilon(h^2).
    \end{equation}

\paragraph*{Finite-difference control of error}
While the estimate $E_i = O_\epsilon(h^2)$is useful, it is not sharp enough for the third-difference calculation. Indeed, $E_i = O_\epsilon(h^2)$ alone implies $\Delta^3E_i = O_\eps(h^2)$, however, it is expected that the $h$-scale variation of neighboring errors produce an additional factor of $h$ with each difference. We therefore refine the preceding argument to obtain
$$\Delta^kE_i = O(h^{k+2})\,\,\,\,\,\,\,\text{for $k=1,2,3.$}$$
Let us start by defining the smooth map $H(.)$
$$H(\theta, E) = \Phi_{\theta, h}(\hat{r}(\theta+h)+E) - \Phi_{\theta, h}(\hat{r}(\theta+h)).$$
Here, $\hat{r}(\theta) = R_0(\theta) + hR_1(\theta)$ and $\hat{r}(\theta_i) \coloneq \hat{r}_i$.\vskip2mm \noindent
Then the exact error can be rewritten as
$$E_i = H(\theta_i, E_{i+1}) + \tau_i.$$
Observe $H(\theta, 0) = 0$. Moreover, the chain rule gives
$$\partial_E H(\theta, E) = \partial_x \Phi_{\theta, h}(\hat{r}(\theta+h)+E).$$
By the localization, since $E_i = O_\epsilon(h^2)$, for sufficiently small $h$, the argument $\hat{r}(\theta+h)+E$ remains in the same uniform neighborhood of the stable limit $R_0(\theta)$. Therefore, we obtain a contraction property for $H$. More specifically
$$|\partial_x \Phi_{\theta, h}(\hat{r}(\theta+h)+E)| \leq \rho_\epsilon.$$
As a consequence, we have
\begin{equation}\label{smoothH_derivative}
|\partial_EH(\theta, E)| \leq \rho_\epsilon. \end{equation}
This, together with the Mean Value Theorem yields
$$|H(\theta, E)| \leq \rho_\epsilon |E|.$$ In addition $\partial_\theta H (\theta, E) = O_\epsilon(E)$. Indeed, by the Fundamental Theorem of Calculus
$$\partial_\theta H(\theta, E) - \partial_\theta H(\theta,0) = \int_0^{|E|} \partial_E\,\partial_\theta H(\theta, s)\,dt$$
\begin{equation}\label{boundedness}
\implies H_\theta(\theta, E) \leq C_\epsilon |E|.
\end{equation}
In the second step, we use the fact that $\partial_\theta H(\theta,0) = 0$ (because $H(\theta,0) = 0$ for every $\theta$) and $H$ is smooth on the uniform contraction neighborhood.\vskip2mm \noindent
Let us now consider $\Delta E_i$. By definition
\begin{align*}
    \Delta E_i & = E_{i+1} -  E_i\\
                   &  = H(\theta_{i+1}, E_{i+2}) + \tau_{i+1} -H(\theta_i, E_{i+1}) - \tau_i\\
                   & = (H(\theta_{i+1}, E_{i+2})- H(\theta_i, E_{i+2})) + (H(\theta_i, E_{i+2})- H(\theta_i, E_{i+1}) + \Delta\tau_i
                \end{align*}
 In the third equation, we add and subtract the term $H(\theta_i, E_{i+2})$ so that we can utilize the estimates for $\partial_\theta H(\theta, E)$ and $\partial_E H(\theta, E)$ obtained above.\vskip2mm \noindent
 The Mean Value Theorem in $\theta$, combined with \eqref{E_i} \& \eqref{boundedness} , results in an $O_\epsilon (h^3)$ bound for the first bracket. Similarly, for the second bracket, by the MVT (in $E$), combined with \eqref{smoothH_derivative}, gives the bound $\rho_\epsilon |\Delta E_{i+1}|$. Finally, for $\Delta \tau_i$ term, use \eqref{difference_tau}. Therefore
\begin{equation}\label{E_i_recuurence}
    |\Delta E_i|  \leq \rho_\epsilon |\Delta E_{i+1}| + C_\epsilon h^3.
\end{equation}
Equation \eqref{E_i_recuurence} has the same structure as the original error estimate calculated above. Therefore, using the same geometric contraction argument, we obtain
\begin{equation*}
    |\Delta E_i|= O_\epsilon (h^3). 
\end{equation*}
Differencing once more, the new non-contractive terms involve bounded derivatives $H_E$ and $H_\theta$ multiplied
by combinations of $|E_i| = O(h^2)$ , $|\Delta E_i|= O_\epsilon (h^3) $, the parameter increment $h$, and $\Delta^2\tau_i = O_\epsilon(h^4)$. Turns out every such term is $O_\epsilon (h^4)$ , and the highest-order error difference again appears with contraction
factor at most $\rho_\epsilon$. Hence
\begin{equation*}
    |\Delta^2 E_i|= O_\epsilon (h^4). 
\end{equation*}
The same argument gives
\begin{equation}\label{Delta_Ei}
    |\Delta^3 E_i|= O_\epsilon (h^5). 
\end{equation}

\paragraph*{Third difference and positivity}
From \eqref{first_order}
\begin{align*}
r_i &= R_0(\theta_i) + hR_1(\theta_i) + E_i\\
&= R_0(\theta_i) \left(1+ h\dfrac{R_1(\theta_i)}{R_0(\theta_i)} + \dfrac{E_i}{R_0(\theta_i}\right)
\end{align*}
Taking the logarithm on both sides 
\begin{equation} \label{logr_i}
\log r_i = \log R_0(\theta_i) + \log\left(1+ h\dfrac{R_1(\theta_i)}{R_0(\theta_i)} + \dfrac{E_i}{R_0(\theta_i)}\right).
\end{equation}
Define
$$L_0(\theta) = \log R_0(\theta) = \log \left(\dfrac{1-\theta}{\theta}\right)  \,\,\,\text{and}\,\,\, L_1(\theta) = \dfrac{R_1(\theta)}{R_0(\theta)} = \dfrac{3\theta-1}{2\theta (1- \theta)}$$
so that $\log r_i = L_0(\theta_i) + \log(1+ hL_1(\theta_i) + e_i)$, where $e_i = \dfrac{E_i}{R_0(\theta_i)}$.\vskip2mm \noindent
             Since $R_0>0$ on our compact interior interval and $E_i = O_\epsilon(h^2)$, we have $e_i = O_\epsilon(h^2)$.\vskip2mm \noindent
             Using $\log(1+x) = x - x^2/2 + O(x^3)$, write
             $$\log(1+ hL_1(\theta_i) + e_i) = hL_1(\theta_i)+ e_i -\dfrac{1}{2}h^2L_1^2(\theta_i) + O_\epsilon(h^3).$$
              Indeed, because $e_i^2 = O_\epsilon(h^4)$ and $hL_1(\theta_i)e_i = O_\epsilon(h^3)$.\vskip2mm \noindent
              Take $\eta_i = \log(1+ hL_1(\theta_i) + e_i) -hL_1(\theta_i)=  e_i -\dfrac{1}{2}h^2L_1^2(\theta_i) + O_\epsilon(h^3)$, so equation \eqref{logr_i} becomes
              \begin{equation}\label{logriII}            
              \log r_i = L_0(\theta_i) + hL_1(\theta_i) + \eta_i.
 \end{equation}
              And it is easy to see $\eta_i = O_\epsilon(h^2)$. Then the smoothness of $R_0, R_1$ and $\log(1+x)$ near $x=0$, together with $\Delta^kE_i = O(h^{k+2})$ for $k=0,1,2,3$, gives
              \begin{equation}
                  \Delta^k \eta_i = O_\epsilon(h^{k+2})\,\,\,\,\,\,\text{for $k=0,1,2,3$.}
              \end{equation}
              In particular, $\Delta^3\eta_i = O(h^5).$\vskip2mm \noindent
              Let us now move to our main equation
              $$D_{m,i} = -\Delta^3 \log r_{i-1}.$$
              Substituting \eqref{logriII}:
              $$D_{m,i} = -\Delta_h^3 L_0(\theta_{i-1}) - h \Delta_h^3 L_1(\theta_{i-1}) - \Delta^3\eta_{i-1},$$
              where $\Delta_hF(\theta) = F(\theta+h)-F(\theta)$, the discrete difference with step size $h$.\vskip2mm
              Since $L_0, L_1 \in C^\infty(0,1)$, the Taylor's formula (see Lemma \ref{finite difference}) provides the following third-difference expansions
              $$-\Delta_h^3 L_0(\theta_{i-1})  = -h^3L_0'''(\theta_{i-1}) + O_\epsilon(h^4)\,\,\,\text{and}\,\,\,\Delta_h^3 L_1(\theta_{i-1}) = O_\epsilon(h^3)$$
             Thus
              \begin{equation}
                  D_{m,i} = -h^3L_0'''(\theta_{i-1})+ O_\epsilon(h^4).
              \end{equation}
              It remains to compute the leading term $L_0'''(\theta_{i-1})$. Since $\theta_{i-1} = \theta_i - h$, $L_0(\theta_{i-1}) = \log(1-\theta_i+h) - \log(\theta_i-h)$, so the direct computation gives
              $$-L_0'''(\theta_{i-1}) = 2\dfrac{3\theta_{i-1}^2-3\theta_{i-1}+1}{\theta_{i-1}^3(1-\theta_{i-1})^3}$$
              which is positive since $3\theta_{i-1}^2-3\theta_{i-1}+1 = 3(\theta_{i-1}-1/2)^2 +1/4>0$.\vskip2mm
               This proves that $D_{m,i}>0$ throughout the compact interior for all sufficiently large m.    
         \subsection{Left-edge regime}
We now consider what we call the \textit{left-edge regime}. This corresponds to the case $\dfrac{i}{m} \to 0$, or equivalently, $i = o(m)$. Note that the argument carried out in the interior regime cannot be extended uniformly to this case because the stable limiting function $R_0(\theta) = \dfrac{1-\theta}{\theta}$ blows up as $\theta \to 0$. More importantly, the contraction factor satisfies $\partial_x \Phi_{\theta,0}(R_0) \to 1$. Thus, the uniform contraction and $\theta$ expansions used above degenerate near the left edge, so a separate analysis is required. \vskip2mm
For convenience, we divide this regime into two cases: \textit{fixed left-edge} and \textit{growing left-edge}.
\subsubsection{Fixed left-edge indices}
    Suppose $i$ remains bounded as $m \to \infty$, i.e. $i = O(1)$. For fixed $i$, Chen-Gu/Zhang bounds \eqref{CGZ_bound} implies
    $$u_i = \dfrac{d_{i+1}\,d_{i-1}}{d_i^2} \to \dfrac{i}{i+1}$$
    Consequently,
    $$e^{D_{m,i}} = \dfrac{u_i^2}{u_{i-1}u_{i+1}} \to \dfrac{i^3(i+2)}{(i+1)^3 (i-1)}$$
    The limiting quantity is strictly greater than $1$ since 
    $$i^3(i+2) - (i+1)^3(i-1) = 2i+1>0.$$
    Hence for fixed $i=O(1)$, $D_{m,i}>0$ for all sufficiently large $m$.
\subsubsection{Growing left-edge}
Suppose $i \to \infty$ while $i = o(m)$. Recall from the interior analysis, the stable limiting function is $ R_0(\theta) = \dfrac{1-\theta}{\theta}=\dfrac{m-i}{i}$. As $\theta \to 0$, this suggests that the dominant growth of $r_i$ is $m/i$. Thus to remove this, introduce the normalized ratio
$$s_i \coloneq \dfrac{i}{m}\,r_i$$
so that $s_i \sim 1$ in the growing left-edge regime.\vskip2mm \noindent
Taking the logarithm of the normalized ratio $s_i$ gives
$$\log r_i = \log s_i  - \log i + \log m.$$
Substituting into $D_{m,i}$
\begin{equation}\label{D_growingleft}
D_{m,i} = \log \left(\dfrac{i^3(i+2)}{(i-1)(i+1)^3}\right) + \log \left(\dfrac{s_{i+1}^3 s_{i-1}}{s_i^3 s_{i+2}}\right).
\end{equation}
Set $\ell_i \coloneq  \log s_i$, so \eqref{D_growingleft} becomes
\begin{equation}
D_{m,i} = \log \left(\dfrac{i^3(i+2)}{(i-1)(i+1)^3}\right)  - \Delta^3(\ell_{i-1}).
\end{equation}
By Lemma \ref{logarithmic expansion}
$$\log \left(\dfrac{i^3(i+2)}{(i-1)(i+1)^3}\right) = \dfrac{2}{i^3}+O(i^4).$$ Thus to prove $D_{m,i}>0$, it is enough to show $\Delta^3(\ell_i)  = o(i^{-3})$.

\begin{remark}
The growing left-edge regime can be separated into two ranges. To see the cut off, recall that $h=1/m$ and $\theta=i/m$. As $\theta \to 0$, the interior expansion contains terms whose size is governed by the factor $\dfrac{h}{\theta^2} = \dfrac{m}{i^2}$. For example, as $\theta \to 0$, the first-order correction satisfies
$$|hR_1(\theta)| \sim \dfrac{h}{2\theta^2} = \dfrac{m}{2i^2}$$
Thus, the transition occurs when $i$ is of order $\sqrt{m}$. This motivates treating separately {\it inner left-edge regime} 
$$i = O(\sqrt{m}).$$ 
and {\it outer left-edge regime}
    $$i/\sqrt{m} \to \infty\,,\, i/m \to 0$$ 

In the inner regime, $m/i^2$ need not be small, so the expansion argument used in the interior cannot be applied directly; instead, we exploit the normalized ratios $s_i$ defined above. In the outer regime, $m/i^2 \to 0$ , allowing an interior-type asymptotic expansion to be recovered with suitably adapted estimates.
\end{remark}
\vskip4mm
\medskip \noindent
{\bf Case I: Inner left-edge regime, $i = O(\sqrt{m})$}\vskip2mm \noindent
Set $$\delta = m^{-1/2}\,\,\,\text{and}\,\,\,t_i = i\,m^{-1/2}$$
so that the regime $i = O(\sqrt{m})$ corresponds to $t_i = O(1)$ while $\delta \to 0$ as $m \to \infty$. More precisely, for any fixed $C>0$, we consider $i \leq C\sqrt{m}$ so that $t_i \leq C$.\vskip2mm \noindent
    As in the interior analysis, we make use of a localization. In particular, by the inner left-edge localization established in Appendix (see Corollary \ref{Localization}(ii))
        $$s_i = 1+ O_C(\delta)$$ uniformly for $i \leq C\sqrt{m}$. Thus all relevant normalized ratios lie in a fixed neighborhood of 1 for sufficiently small $\delta$. The localization alone, however, does not automatically give sufficient finite-difference control. So the next step is to utilize the Boros-Moll recurrence for $s_i$ to control its variation across neighboring indices.\vskip2mm \noindent
      Substitute $r_i = (m/i)s_i$ into \eqref{BM_recurrence}:
      \begin{equation}
          (m+1-i)(m+i) - (2m+1)ms_i + m^2 s_is_{i+1} = 0.
      \end{equation}
      Solving for $s_i$
      \begin{equation*}
          s_i = \dfrac{(m+1-i)(m+i)}{(2m+1) - m^2s_{i+1}} = \dfrac{m^2 - i^2 +m +i}{(2m+1) - m^2s_{i+1}}
      \end{equation*}
      After dividing both the denominator and numerator by $m^2$ and substituting $\delta$ and $t_i$, we get
      \begin{equation}\label{backward_Psi}
          s_i = \dfrac{1 + (1-t_i^2)\delta^2 + t_i\delta^3}{2+ \delta^2 - s_{i+1}}.
      \end{equation}
      Equivalently $s_i = \Psi_{\delta, t_i}(s_{i+1})$, where $\Psi_{ \delta, t}(x) =\dfrac{1 + (1-t^2)\delta^2 + t\delta^3}{2+ \delta^2 - x}$; the backward map for the normalized ratio $s_i$. \vskip2mm \noindent
      For $t_i \leq C$, the estimate $s_i = 1+ O_C(\delta)$ shows that all relevant $s_i$'s stay in a neighborhood of 1. Hence for sufficiently small $\delta$, the denominator $2+ \delta^2 - x$ is uniformly bounded away from $0$ for $t\leq c$ and $x$ in this neighborhood. Therefore, $\psi_{ \delta, t}$ is smooth on the relevant neighborhood and all of its derivatives required below are bounded uniformly. \vskip2mm \noindent
      Let $\mu_i \coloneq s_i - 1$. From \eqref{backward_Psi}
      \begin{equation}\label{mu_diff}
          \mu_i - \mu_{i+1} = \mu_i \mu_{i+1}- t_i^2 \delta^2 + t_i \delta^3-\mu_i\delta^2.
          \end{equation}
      By definition $\mu_i = O_C(\delta)$ and $t_i = O_C(1)$. Therefore, $\mu_i \mu_{i+1} = O_C(\delta^2) , t_i^2\delta^2 = O_C(\delta^2) , t_i\delta^3 = O_C(\delta^3)$ and $\mu_i\delta^2 = O_C(\delta^3)$. Substituting in \eqref{mu_diff}
      $$\Delta\mu_i = O_C(\delta^2).$$
      Taking a finite difference in \eqref{mu_diff} and using the fact that $t_{i+1}- t_i = \delta$, we find that each term on the RHS gains an additional factor $\delta$. More specifically,
      $\Delta \mu_i \mu_{i+1} = O_C(\delta^3), \Delta (t_i^2\delta^2) = O_C(\delta^3) , \Delta (t_i\delta^3) = \delta^3 \Delta t_i = O_C(\delta^4)\,\,\,\text{and}\,\,\,\Delta (\mu_i\delta^2) = \delta^2 \Delta \mu_i = O_C(\delta^4)$
      Therefore
      $$\Delta^2\mu_i = O_C(\delta^3).$$
      Similarly, taking another difference and using $\mu_i = O_C(\delta) , \Delta\mu_i = O_C(\delta^2)$ and $\Delta^2\mu_i = O_C(\delta^3)$ gives
       $$\Delta^3\mu_i = O_C(\delta^4).$$
              Finally, since $s_i=1+\mu_i$ and $\mu_i=O_C(\delta)$, the values $s_i$ remain uniformly bounded away from 0. Moreover, the function $x \mapsto \log(1+x)$ is smooth near $x=0$ with uniformly bounded derivatives on the relevant neighborhood. Combining that with discrete chain and product rules, together with $\Delta^k\mu_i = O_C(\delta^{k+1})$ for $k=0,1,2,3$ gives
       $$\Delta^k \ell_i= O_C(\delta^{k+1}).$$
       In particular, $\Delta^3 \ell_i= O_C(\delta^4) =O_C(m^{-2})$.\vskip2mm \noindent
       Now as $i \to \infty$ and $i \leq C\sqrt{m}$
       $$\dfrac{m^{-2}}{i^{-3}} = \dfrac{i^3}{m^2} \leq \dfrac{C^3}{\sqrt{m}} \to 0$$
       and hence $\Delta^3 \ell_i = o(i^{-3})$.
       \vskip5mm \noindent
{\bf Case II: Outer left-edge regime, $\dfrac{i}{\sqrt{m}} \to \infty$ and $\dfrac{i}{m} \to 0$} \vskip2mm \noindent
We now consider the left regime where $\dfrac{i}{\sqrt{m}} \to \infty$ and $\dfrac{i}{m} \to 0$. This is equivalent to,
$$\theta_i = \frac{i}{m} \to 0\,\,\text{and}\,\,\frac{h}{\theta_i^2} = \frac{m}{i^2} \to 0$$
Recall the functions $R_0(\theta) = \dfrac{1-\theta}{\theta}$ and $R_1(\theta) = \dfrac{3\theta-1}{2\theta^2} $ obtained in the interior analysis. With this, we reuse 
$$\hat{r}_i = R_0(\theta_i) + hR_1(\theta_i) $$ as a candidate approximation in the outer left-edge regime. The interior error estimates, however, are not uniform as $\theta \to 0$. Consequently, the validity of this approximation must be established again, with  dependence on $\theta$. This is where the condition $h/\theta^2 \to 0$ becomes useful as it will ensure that the correction and error terms obtained below are asymptotically small.
\vskip2mm \noindent
\paragraph*{Two-term approximation and residual estimates.}

Recall the one-step residual
$$ \tau_i \coloneq \Phi_{\theta_i, h}(\hat{r}_{i+1}) - \hat{r}_i.$$
The direct computation yields
\begin{equation}\label{explicit_tau}
    \tau(\theta,h)= \dfrac{-h^2 (3h\theta-h + 3\theta^2 + 4\theta-3)}{2\theta^2 (h^2+3h\theta+ 3h + 2\theta^2 + 2\theta)}.
\end{equation}
Since $h/\theta^2 \to 0$, $h = o(\theta^2)$. But $\theta^2 = o(\theta)$, and thus $h = o(\theta)$. Let $Q(\theta,h) \coloneq h^2+3h\theta+ 3h + 2\theta^2 + 2\theta$. 
$$\frac{Q(\theta,h)}{\theta} = \dfrac{h^2}{\theta} + 3h + 3\frac{h}{\theta} + 2\theta + 2.$$
Since $h^2/\theta = h\cdot h/\theta $ and both factors approach 0, we have $h^2/\theta \to 0$. Similarly, $h = \theta \cdot h/\theta \to 0$ as both $\theta, h/\theta \to 0$. Therefore,
$$\frac{Q}{\theta} = 2 + o(1).$$ which implies $Q \asymp \theta.$
As a result, the denominator of \eqref{explicit_tau} has size $O(\theta^3)$. \vskip2mm \noindent
For the numerator, notice that as $h, \theta \to 0$, $3h\theta-h + 3\theta^2 + 4\theta-3 \to -3$. Therefore, it is bounded. Consequently, the numerator has size $O(h^2)$. We get
$$\tau(\theta,h) = O(h^2 \theta^{-3}).$$
The expression \eqref{explicit_tau} also allows us to differentiate with respect to $\theta$. First, since $Q(\theta,h) = \theta q(\theta.h)$, where $q(\theta, h) = 2+ 2\theta+ 3h + \dfrac{h^2}{\theta}+ 3\dfrac{h}{\theta}$, the denominator of \eqref{explicit_tau} can be written as $2\theta^3\, q(\theta,h)$. Also, let $P(\theta,h) =3h\theta-h + 3\theta^2 + 4\theta-3$. Then
$$\tau(\theta, h ) = -\dfrac{h^2 P(\theta,h)}{2\theta^3 \,q(\theta,h)}.$$
Define $G(\theta,h) = -\dfrac{P(\theta,h)}{2q(\theta,h)}$. Then $\tau(\theta,h) = h^2\theta^{-3}G(\theta,h)$. With $h/\theta^2 \to 0$, all the terms occurring in $G$ are small. Also, $q(\theta,h)$ is bounded away from 0. More importantly, differentiating $G$ w.r.t $\theta$ costs at most one power of $\theta^{-1}$ each time. For example,
$$ \partial_\theta (h/\theta) =-h/\theta^2 \,\,\text{and}\,\,\partial_\theta^2 (h/\theta) =2h/\theta^3$$and similarly,
$$\partial_\theta^k (h^2/\theta) = O(h^2\theta^{-k-1})\,\,\,k=0,1,2,3.$$
Since $h/\theta^2$ is bounded (in fact it tends to zero), these terms satisfy bounds of the form
\begin{equation}\label{partial_G}
\partial_\theta^k G = O(\theta^{-k})\,\,\,k=0,1,2,3.
\end{equation}
Let us now differentiate $\tau = h^2\theta^{-3}G$ w.r.t $\theta$. By Leibniz rule
$$\partial_\theta^k \tau = h^2 \sum_{i=0}^k \binom{k}{i}\partial_\theta^i (\theta^{-3}) \partial_\theta^{k-i}G.$$
It follows from \eqref{partial_G} that $\partial_\theta^{k-i} G = O(\theta^{-(k-i)})$ and $\partial_\theta^i (\theta^{-3}) = O(\theta^{-3-i}) $. Thus every summand is $O(\theta^{-k-3})$. This implies, in the region $h/\theta^2 = o(1)$,
\begin{equation}\label{partial_tau}
\partial_\theta^k \tau(\theta,h) = O(h^2\theta^{-k-3}) \,,k=0,1,2,3.\end{equation}
Then by the Fundamental Theorem of Calculus
$$\Delta \tau_i = \tau(\theta_i+h,h) - \tau(\theta_i,h) = \int_0^h \partial_\theta \tau (\theta_i+t,h)\,dt.$$
Combining this with \eqref{partial_tau} gives 
$$\implies \Delta \tau_i  = O(h^3 \theta_i^{-4}).$$
And the repeated application of FTC allows us to conclude
\begin{equation} \label{tau_diff_theta}
     \Delta^k \tau_i  = O(h^{k+2} \theta_i^{-k-3})\,,k=0,1,2,3.
\end{equation}
This is the analogue of the interior residual estimate \eqref{difference_tau}, except that the $\theta$-dependence captures its explicit deterioration as $\theta \to 0$.\vskip2mm
\paragraph*{Localization, contraction, and error estimates.}
We will now estimate the error. Write
$$E_i = r_i - \hat{r}_i$$ In terms of the backward map $\Phi_{\theta_i, h}$ 
$$E_i = \Phi_{\theta_i,h}(r_{i+1}) - \Phi_{\theta_i, h}(\hat{r}_{i+1}) + \tau_i .$$
Recall that in the interior regime we used the uniform contraction factor $\rho_\epsilon<1$, however, such a uniform constant is unavailable here, since
$$\partial_x \Phi_{\theta,0}(R_0(\theta)) = \dfrac{1-\theta}{1+\theta} = 1-2\theta + O(\theta^2) \to 1.$$
Thus the contraction gap is only of order $\theta$, and we need an estimate that is correspondingly sharper, and as before we turn to localization. More specifically, by the outer left-edge localization established in Appendix (see Corollary \ref{Localization}(iii)),
$$\frac{r_i- R_0(\theta_i)}{R_0(\theta_i)} = O(i^{-1}) = o(\theta_i).$$
We next verify the same relative localization for the approximate ratio $\hat{r}_i$.
\vskip2mm \noindent
       Since $\hat{r}_i = R_0(\theta_i) + hR_1(\theta_i)$, we have
       $$\dfrac{\hat{r}_i-R_0(\theta_i)}{R_0(\theta_i)} = \dfrac{hR_1(\theta_i)}{R_0(\theta_i)} = \dfrac{3i-m}{2i(m-i)} = -\dfrac{1}{2i}+ \dfrac{1}{m-i}$$
       $$\implies \dfrac{\hat{r}_i}{R_0(\theta_i) } = 1 -\dfrac{1}{2i} + \dfrac{1}{m-i} = 1+O(i^{-1})\,\,\,\,\,\,(\text{because $i = o(m)$})$$
       Thus
       $$\dfrac{\hat{r}_i- R_0(\theta_i)}{R_0(\theta_i) } = o(\theta).$$
       Consequently, both $r_i$ and $\hat{r}_i$ lie in an $o(\theta_i)$- relative neighborhood of $R_0(\theta_i)$. Hence every point $x$ on the segment joining them has the form
       $$x = R_0(\theta_i)(1+ o(\theta_i)).$$
       We therefore proceed to evaluate the derivative of the backward map $\Phi_{\theta,h}$ uniformly on this segment.             
       \vskip2mm \noindent
       Going back to $\Phi_{\theta,h}$, let us write
       $$\partial_x
       \Phi_{\theta,h}(x) = A_{\theta,h} (B_{\theta,h} - x)^{-2}\,,$$
       where $A_{\theta,h} = \dfrac{(1-\theta+h)(\theta +1)}{\theta (\theta+h)}$ and $B_{\theta,h} = \dfrac{2+h}{\theta+h}$.\vskip2mm
       Let $x=R_0(\theta)(1+ o(\theta))$ so its relative perturbation is $o(\theta)$. We compare $x$ with $R_0(\theta)$. First, $R_0(\theta) \asymp \theta^{-1}$, and hence, $x-R_0(\theta) = R_0(\theta)\,o(\theta) = O(\theta^{-1})\,o(\theta) = o(1)$. On the other hand, $B_{\theta,h}- R_0(\theta) \asymp \theta^{-1}$. Therefore,
       $$\dfrac{x - R_0(\theta)}{B_{\theta,h}- R_0(\theta) } = o(\theta) \implies B_{\theta,h}-x = (B_{\theta,h}- R_0(\theta))  (1+ o(\theta)).$$
       Taking the inverse square gives
       $$(B_{\theta,h}-x)^{-2} = (B_{\theta,h}- R_o(\theta))^{-2}  (1+ o(\theta)).$$
       Hence
       \begin{equation}\label{partialx_Phi}
           \partial_x \Phi_{\theta,h}(x) = \partial_x \Phi_{\theta,h}(R_0(\theta))(1+ o(\theta)).
       \end{equation}
       Direct computation gives 
       \begin{align*}
           \partial_x \Phi_{\theta,h}(R_0(\theta)) &= \dfrac{(1-\theta+h)(1+\theta) \theta (\theta+h)}{[\theta+ \theta^2+ h(2\theta-1)]^2}\\
           & = \dfrac{(1-\theta+h)(1+\theta) (1+h/\theta)}{[1+ \theta+ h/\theta(2\theta-1)]^2}
           \end{align*}
           Since $\dfrac{h}{\theta^2} \to 0$, $\dfrac{h}{\theta} = o(\theta)$. Also $h = o(\theta)$. Consequently,
           \begin{equation}
               \partial_x \Phi_{\theta,h}(R_0(\theta)) = 1-2\theta + o(\theta).
           \end{equation}
           Using \eqref{partialx_Phi}
           $$\partial_x \Phi_{\theta,h}(x)  = 1-2\theta + o(\theta).$$           
This shows that perturbing $x$ away from $R(\theta)$ changes $\partial_x \Phi_{\theta,x}$ by $o(\theta)$. Therefore, uniformly for any $x$ between $r_{i+1}$ and $\hat{r}_{i+1}$
$$\partial \Phi_{\theta_i,h}(x) = 1- 2\theta_i+ o(\theta_i).$$
In particular, throughout a sufficiently small interval $(o, \vartheta_0)$, provided $m$ is sufficiently large
\begin{equation}
    |\partial \Phi_{\theta_i,h}(x) | \leq 1- c\theta,
\end{equation}
where $c>0$ is fixed.
\vskip2mm \noindent
Then the Mean Value Theorem, combining with \eqref{tau_diff_theta} yields
\begin{equation} \label{recurrenceE_i_outerleft}
    |E_i| \leq (1- c\theta_i)|E_{i+1}|+ Ch^2\theta_i^{-3}.
\end{equation}
Unlike in the interior regime, the contraction here is not uniform, since $1-c\theta_i\to1$ as $\theta_i\to0$. To control the iteration, we combine this recurrence with the previously established interior estimates.
\vskip2mm \noindent
Choose a fixed $\vartheta \in (0, \vartheta_0)$ such that $\theta = \vartheta$ lies inside a compact interval on which the interior analysis applies, and set $J = \lfloor \vartheta m \rfloor$
so that $\theta_J = \dfrac{J}{m} =  \vartheta + O(h)$ which lies in the earlier interior regime. In particular, $E_J = O_\vartheta(h^2)$.\vskip2mm \noindent
Thus, from $i$ through $J-1$, we use the recurrence \eqref{recurrenceE_i_outerleft}, while at the terminal index $J$, we use the already proved interior estimate $E_J = O_\vartheta(h^2)$.\vskip2mm \noindent
Iterating \eqref{recurrenceE_i_outerleft} from $i$ through $J-1$ results in
\begin{equation}
|E_i| \leq \prod_{j=i}^{J-1} (1- c\theta_j)|E_J| + Ch^2 \sum_{k=i}^{J-1}\theta_k^{-3} \prod_{j=i}^{k-1}(1- c\theta_j).
\end{equation}
Using $1-x \leq e^{-x}$,
$$\prod_{j=i}^{J-1} (1- c\theta_j) \leq \exp \left(-c\,\sum_{j=i}^{J-1}\theta_j\right).$$
But $\theta_j = \dfrac{j}{m}$, so $\displaystyle \sum_{j=i}^{J-1} \theta_j = \dfrac{1}{m} \sum_{j=i}^{J-1} j = \dfrac{(J-i)(J+i-1)}{2m}$. With $J \sim \vartheta m$ and $i = o(m)$, we have $\displaystyle \sum_{j=i}^{J-1} \theta_j = \dfrac{\vartheta^2}{2}m + o(m)$. Hence, for some $c_\vartheta>0$
$$\prod_{j=i}^{J-1} (1- c\theta_j) \leq e^{-c_\vartheta m}.$$ Together with $E_J = O_\vartheta(h^2)$, we get
$\prod_{j=i}^{J-1} (1- c\theta_j)|E_J| = O_\vartheta (h^2 e^{-c_\vartheta m})$.\vskip2mm \noindent
For the residual term \noindent
\begin{align*}
    \sum_{k=i}^{J-1}\theta_k^{-3} \prod_{j=i}^{k-1}(1- c\theta_j) \leq \theta_i^{-3} \sum_{j=0}^\infty (1-c\theta_i)^n = \theta_i^{-3} \dfrac{1}{c\theta_i}= \dfrac{\theta_i^{-4}}{c}.
\end{align*}
Combining both estimates, 
$$|E_i| \leq  O_\vartheta (h^2 e^{-c_\vartheta m}) + O(h^2 \theta^{-4})$$
and hence $$E_i = O(h^2 \theta^{-4}).$$ as the first term is negligible for sufficiently large $m$.\vskip2mm \noindent
This also confirms the validity of the error approximation
\begin{equation}\label{approx_outer_left}
    r_i = R_0(\theta_i) + hR_1(\theta_i) +O(h^2 \theta_i^{-4}). 
    \end{equation}
The higher-difference estimates $\Delta^k E_i$ are obtained exactly as in the interior regime. The only modification is that the uniform contraction factor $\rho_\epsilon<1$ is replaced by the variable factor $1-c\theta_i$. Applying one forward difference to the error recurrence, and then iterating inductively, gives
\begin{equation*}
    |\Delta^kE_i| \leq (1- c\theta_i)|\Delta^kE_{i+1}|+ Ch^{k+2}\theta_i^{-k-3}\,,\,\,0\leq k \leq 3.
\end{equation*}
After that. iterating through $J=\lfloor\vartheta m\rfloor$, exactly as above, introduces one additional factor $\theta_i^{-1}$. Thus
\begin{equation}
    \Delta^k E_i = O(h^{k+2}\theta_i^{-k-4})\,,\,\,\,k=0,1,2,3.
\end{equation}
\paragraph*{Third difference and positivity.}
It remains to pass this to $\log r_i$. To this end, we return to the error estimate and obtain
$$r_i = \hat{r}_i + E_i = \hat{r}_i\left(1+ u_i\right)\,,\,\,\,\,\text{where $u_i = \dfrac{E_i}{\hat{r}_i}$}$$
Using the two-term approximation, we also have
$$\hat{r}_i = R_0(\theta_i) \left(1+ h\dfrac{R_1(\theta_i)}{R_0(\theta_i)}\right).$$
It is easy to see that $h\dfrac{R_1(\theta_i)}{R_0(\theta_i)} =\dfrac{3i-m}{2i(m-i)} = -\dfrac{1}{2i}+ \dfrac{1}{m-i}.$\vskip2mm \noindent
After combining these, we get the exact factorization
\begin{equation}\label{factor_ri}
    r_i = R_0(\theta_i) \left(1 -\dfrac{1}{2i}+ \dfrac{1}{m-i}\right)(1+ u_i).
\end{equation}
Taking the logarithm on both sides gives
\begin{equation} \label{logri_factor}
    \log r_i = \log R_o(\theta_i) + \log \left(1 -\dfrac{1}{2i}+ \dfrac{1}{m-i}\right) + \log(1+u_i).
\end{equation}
This decomposition lets us finish the error analysis. In particular, applying $-\Delta^3$ to \eqref{logri_factor} , with the index shift $i \mapsto i-1$, and using Lemma \ref{third difference}, we obtain
\begin{align*}
    -\Delta^3 \log R_0(\theta_{i-1}) &= \dfrac{2}{i^3} + o(i^{-3}).\\
    -\Delta^3 \log \left(1 -\dfrac{1}{2(i-1)}+ \dfrac{1}{m-(i-1)}\right) &= o(i^{-3}).\\
   - \Delta^3 \log(1+ u_{i-1})&= o(i^{-3}).
    \end{align*}
\vskip2mm \noindent
Thus we conclude
$$D_{m,i} = \dfrac{2}{i^3}+ o(i^{-3})>0.$$
for sufficiently large $m$ in the outer-left regime.

\subsection{Right-edge regime}

Finally, we treat the \textit{right-edge regime}, which corresponds to the case $\dfrac{i}{m} \to 1$, or equivalently, if we write $j = m-i$, then $j = o(m)$. We first look at the fixed right-edge indices.
\subsubsection{Fixed right-edge indices: $j=O(1)$}
The argument is similar to the fixed left-edge case, i.e. use Chen-GU/Zhao bounds.\vskip2mm \noindent
Recall
\begin{equation*}
        f_i(m) < u_i(m) < g_i(m)\,,\end{equation*}
       where $f_i(m) = \dfrac{(m-i)i}{(m - i+1)(i+1)}$ and $g_i(m) = f_i(m)\left(1+ \dfrac{1}{m+i^2}\right)$. \vskip2mm
       Set $i = m-j$, then $f_{m-j}(m) = \dfrac{j}{j+1}\cdot \dfrac{m-j}{m-j+1}$ and $g_{m-j} = \dfrac{j}{j+1}\cdot \dfrac{m-j}{m-j+1}\cdot \left(1+ \dfrac{1}{m+ (m-j)^2}\right)$. Notice that, for fixed $j$, $f_{m-j} , g_{m-j} \to \dfrac{j}{j+1}$ as $m \to \infty$. Thus
       $$u_{m-j} \to \dfrac{j}{j+1}.$$
       Similarly, replacing $j$ with $j-1$ and $j+1$ results in
       $$u_{m-j+1} \to \dfrac{j-1}{j} \,\,\,\text{and}\,\,\,\, u_{m-j-1} \to \dfrac{j+1}{j+2}$$
       Therefore,
       $$e^{D_{m, m-j}} = \dfrac{u_{m-j}^2}{u_{m-j-1}\, u_{m-j+1}} = \dfrac{j^3(j+2)}{(j-1)(j+1)^3}$$
       As shown in the fixed left-edge case, $\dfrac{j^3(j+2)}{(j-1)(j+1)^3}>1$, which means for each fixed $j \geq 3$, $D_{m,m-j} >0$ for all sufficiently large $m$.

       \subsubsection{Growing right-edge indices: $j \to \infty$ and $j = o(m)$}
       Let $y_j = \dfrac{j}{m}$, and define the normalized ratio
    
$$q_j \coloneq \dfrac{m}{j+1}r_{m-j}.$$       
       Then in terms of $q_j$, the Boros-Moll recurrence \eqref{BM_recurrence} becomes
       $$q_j = \X_{y_j,h}(q_{j-1})\,,$$
       where $\X_{y,h}(q) = \dfrac{2-y}{(1-y)[2+h-y(1-y+h)q]}$.\vskip2mm
       The normalization $r_{m-j} = \dfrac{j+1}{m}q_j$ captures the dominant right-edge behavior. Indeed, with $\theta=1 - \dfrac{j}{m}$ , the stable limiting function satisfies $R_0(\theta) = \dfrac{1-\theta}{\theta} = \dfrac{j}{m-j} \sim j/m$ since $j = o(m)$. The normalization is therefore designed so that $q_j$ has a limiting value 1.   \vskip2mm
       Recall at the left endpoint, the derivative of the backward map along $R_0(\theta)$ approaches $1$, so the contraction degenerates and one needs and approximation of the form $R_0+hR_1$ in order to make the residual sufficiently small. Here, by contrast,
       $$\partial_q \X_{y,h}(q) = \dfrac{(2-y)y(1-y+h)}{(1-y)[2+h-y(1-y+h)q]^2} = O(y).$$for $q$ in a fixed neighborhood of 1, and therefore, $\partial_q \X_{y,h}(q) \to 0$ as $y,h \to 0$. \vskip2mm
       Thus the normalized dynamics becomes increasingly contracting near the right edge, so it is good enough to control the finite differences of $q_j$ without constructing an approximation. This makes the analysis for the growing right edge indices considerably simpler than the growing left edge case. \vskip5mm
       \paragraph*{Localization and contraction.}

       As in previous cases, we need a localization to trap the varying iterates $q_j$ in a uniform region where contraction, derivative bounds, iteration, and the logarithmic estimate all hold simultaneously.\vskip2mm
       Note that the end point value (at $j=0$) satisfies $q_0 = mr_m = \dfrac{2m}{2m+1} = 1 + O(h)$, while $\X_{y,h}(q) \to 1$ as $y,h \to 0$ for $q$ in a fixed neighborhood of 1. \vskip2mm \noindent
       Also, for fixed $C>0$
       \begin{equation}\label{uppercase_Chi}
           \X_{y,h}(q) = 1+ O_C(y+h)\end{equation} uniformly for $|q|\leq C$. Indeed, writing
       $$N \coloneq 2-y \,\,\text{and}\,\, D \coloneq (1-y)[2+h-y(1-y+h)q],$$
       we have
       $$D-N = h - y -hy -y(1-y)(1-y+h)q = O_C(y+h).$$ Moreover $$D = 2+ O_C(y+h)$$ so for sufficiently small $y+h$, $D\geq 1$. Hence $$|\X_{y,h}(q)-1 | = \left|\dfrac{D-N}{D}\right| \leq |D-N| = O_C(y+h).$$
       Let us now choose a compact interval $I = [a,b] \subseteq (0, \infty)$ such that  $1 \in (a,b)$. Since $q_0 \to 1$, $q_0 \in I$ for sufficiently large $m$. Also, from \eqref{uppercase_Chi}, after choosing $\delta, h>0$ sufficiently small
       $$\X_{y,h}(I) \subset I\,\,\,\,\,\,\text{for $0 \leq y \leq\delta$ and $0 < h \leq h_0$ .}$$
       Indeed, since $\X_{y,h}(q)$ is uniformly close to 1 for $q \in I$. \vskip2mm \noindent
       Then the recurrence $q_j = \X_{y_j,h}(q_{j-1})$ gives inductively
       $$q_j \in I\,,\,\,\,0 \leq j \leq  \delta m.$$
              This localization now allows us to place $q_j's$ in a single neighborhood where all subsequent estimates are uniform. In particular, on $[0, \delta] \times [0,h_0] \times I$, the denominator of $\X_{y,h}$ is bounded away from 0, and thus every partial derivative of $\X_{y,h}$ of order at most three used below is uniformly bounded. Then $\partial_q \X_{y,h} (q) = O(y)$ uniformly on this localized region. Therefore, we can choose $0<\rho<1$ such that
       \begin{equation} \label{partial_chi}
           |\partial_q \X_{y,h}(q)| \leq \rho 
            \end{equation}           
           for $0 \leq y \leq \delta\,,\,q\in I$ and for sufficiently small $h$.
\vskip3mm
\paragraph*{Finite-difference estimates.}
           
           Let us show that \begin{equation} \label{deltak_q}
           \Delta^k q_j = O(h^k)\,\,\text{for $k=1,2,3$}
       \end{equation}
       For $k=1$
\begin{align*}
    \Delta q_j &= q_{j+1} - q_j\\
    & = \X_{y_j+h,h}(q_j) -  \X_{y_j,h}(q_{j-1})\\
    & = (\X_{y_j+h,h}(q_j) - \X_{y_j+h,h}(q_{j-1})) + (\X_{y_j+h,h}(q_{j-1}) - \X_{y_j,h}(q_{j-1})).
\end{align*}
The first bracket can be controlled by the Mean Value Theorem, together with contraction \eqref{partial_chi}. More specifically
$$|\X_{y_j+h,h}(q_j) - \X_{y_j+h,h}(q_{j-1})| \leq \rho |\Delta q_{j-1}|.$$
The second bracket is simply $O(h)$ since $\partial_y \X_{y,h}$ is uniformly bounded on the localized region. Hence,
$$|\Delta q_j| \leq \rho |\Delta q_{j-1}| + Ch.$$
Iterating yields
\begin{equation} \label{delta_q}
    \Delta q_j = O(h).
\end{equation}
Let us note that, to conclude $\Delta q_j = O(h)$, we also need $\Delta q_0 = O(h)$, but this follows easily since $q_0 = 1+ O(h)$ and $q_1 = \X_{h,h}(q_o) = 1+ O(h)$.\vskip2mm
Finally, the higher order estimates follow from repeating the preceding mean value and contraction argument after taking one and two further finite differences, and using the smoothness of \(X_{y,h}\) in the recurrence $q_j=\X_{y_j,h}(q_{j-1}),$ together with the uniform contraction \eqref{delta_q}, the uniform derivative bounds established above, and the corresponding initial finite-difference estimates, we obtain
       \begin{equation*}
    \Delta^k q_j = O(h^k)\,\,\,\,\text{for $k=2,3$.}
\end{equation*}
Indeed, at each stage, the highest-order difference is multiplied by the same contraction factor, while the remaining terms contain only lower-order differences already controlled and factors of $\Delta y_j = h$.
\vskip3mm
\paragraph*{Third difference and positivity.}
To pass to $\log q_j$, note that since $q_j \in I$, the derivatives of $\log x$ through order three are uniformly bounded on the range of the $q_j\,'s$. Then it follows from the finite-difference Taylor expansion of $\log x$ on $I$
$$\Delta^3 \log q_j = O(\Delta^3 q_j)+ O((\Delta q_j)(\Delta^2 q_j)) + O((\Delta q_j)^3).$$
Using \eqref{deltak_q}, we conclude
\begin{equation}
    \Delta^3 \log q_j = O(h^3) = O(m^{-3}).
\end{equation}
We can now consider $D_{m,m-j}$. Since $r_{m-j} = \dfrac{j+1}{m}q_j$, after taking logarithm, we get
$$\log r_{m-j} = \log(j+1) - \log m + \log q_j.$$ Substituting into $D_{m,i}$ with $i=m-j$ gives
$$D_{m, m-j} = \log \left(\dfrac{j^3 (j+2)}{(j-1)(j+1)^3}\right) + \log \left(\dfrac{q_{j+1}^3 q_{j-1}}{q_j^3 q_{j+2}}\right).$$
We know $\log \left(\dfrac{j^3 (j+2)}{(j-1)(j+1)^3}\right) = \dfrac{2}{j^3} + O(j^{-4})$ (see appendix). Also $\log \left(\dfrac{q_{j+1}^3 q_{j-1}}{q_j^3 q_{j+2}}\right) = \Delta^3 \log q_{j-2}$. Therefore,
$$D_{m, m-j} = \dfrac{2}{j^3} + O(j^{-4}) + O(m^{-3}).$$
Since $j \to \infty$, we have $O(j^{-4}) = o(j^{-3})$. Moreover, $m^{-3} = j^{-3 \cdot}\dfrac{m^{-3}}{j^{-3}}$ and $\dfrac{m^{-3}}{j^{-3}} \to 0$ since $j = o(m)$. Thus, $O(m^{-3}) = o(j^{-3})$. Combining the two estimates simplify above equation to
\begin{equation*}
    D_{m, m-j} = \dfrac{2}{j^3} + o(j^{-3}).
\end{equation*}
which is positive for sufficiently large $m$ in the growing right-edge regime.
\subsection{Proof of eventual strict log-concavity}
To complete the proof, suppose on contrary that there exists a sequence $m_n \to \infty$ and indices $3 \leq i_n \leq m_n-3$ for which $D_{m_n, i_n}<0$. After passing to a subsequence, $i_n/m_n\to\theta\in[0,1]$. If $0<\theta<1$, the interior result applies. If $\theta=0$, the fixed, inner, and outer left-edge arguments cover all possibilities after further subsequences according to the behavior of $i_n$ and $i_n/\sqrt{m_n}$. If $\theta=1$, the right-edge result applies. More specifically, set $j_n = m_n - i_n$. After passing to a further subsequence, either $j_n$ remains bounded, in which case the fixed right-edge argument applies, or $j_n\to\infty$. In the latter case $j_n=o(m_n)$, so the growing right-edge argument applies. In either case, $D_{m_n,i_n}>0$ eventually, a contradiction.

\section{Final Remarks}
We conclude with two directions suggested by the present work.

\begin{itemize}

\item   The proof of Theorem \ref{eventual_LC} is stated asymptotically and does not provide an explicit value of $M$. The estimates in our analysis arise from explicit rational functions and quantitative contraction arguments, which might suggest that an effective version could be obtained by tracking the constants throughout the interior, left-edge and right-edge regimes. Such an effective threshold, combined with a verification of the remaining finite cases, would fully resolve Conjecture \ref{LC_BM_conj}.

    \item Beyond the current application to the log-concavity conjecture, we believe the ideas and techniques used in the proof of Theorem \ref{eventual_LC} could be utilized to develop a more general asymptotic framework for deriving higher-order shape inequalities from nonlinear ratio recurrences. More specifically, the combination of stable fixed points, contraction, finite-difference estimates, and separate regime analysis provides a discrete dynamical system viewpoint that may be useful in studying other recursively defined combinatorial sequences and arrays.

\end{itemize}
\vskip10mm
\section*{Statement on the Use of AI} The author acknowledges the use of ChatGPT (GPT-5.6 Sol) as a tool for exploring proof strategies for Theorem \ref{eventual_LC}, including the discrete dynamical-system formulation, localization arguments and for refining the exposition of several intermediate estimates. All mathematical statements and proofs in the current manuscript were checked and written by the author, who takes full responsibility for the content of the work.
\newpage
\appendix
\section{Appendix}

\begin{lemma}[Finite differences of smooth functions] \label{finite difference}

Let $F \in C^4(I)$, where $I \subseteq (0,\infty)$ is an interval, and suppose $F^{(4)}$ is uniformly bounded on a compact subset $E \subset I$. Then, for $\theta, \theta+h, \theta+2h, \theta + 3h \in E$,

$$\Delta_h^3F(\theta) = h^3F'''(\theta) + O_E(h^4)\,,$$
where $$\Delta^3_h F(\theta) = F(\theta+3h) - 3F(\theta+2h) + 3F(\theta+h) - F(\theta).$$
\end{lemma}
\begin{proof}
  By Taylor's formula
    $$F(\theta+kh) = F(\theta)+khF^\prime(\theta) + \dfrac{k^2h^2}{2}F''(\theta) + \dfrac{k^3h^3}{6}F'''(\theta) + O_E(h^4).$$
    Substitution into the third difference gives
   \begin{align*}
       \Delta^3_h F(\theta) = (1-3+3-1)F(\theta) + (3-6+3)hF'(\theta) +\left(\dfrac{9}{2}-6 +\dfrac{3}{2}\right)h^2F''(\theta)&\\+\left(\dfrac{9}{2}-4+\dfrac{1}{2}\right)h^3 F'''(\theta)+ O_E(h^4).
   \end{align*}
    The first three terms vanish, and we are left with
    $$\Delta^3_h F(\theta) = h^3 F'''(\theta)+ O(h^4).$$
\end{proof}
\vskip3mm \noindent
\begin{lemma}[Telescoping estimate] \label{telescoping}

Let $P_i \coloneq \displaystyle \prod_{k=i}^{m-1}\left(1+ \dfrac{1}{m+k^2}\right)$ and $U_i \coloneq \dfrac{2m(m-i+1)}{i(2m+1)}$. Then
\begin{equation}\label{telescope_CGZ}
    \dfrac{U_i}{P_i}<r_i< U_i.
    \end{equation}
\end{lemma}
\begin{proof}
    The result follows directly from Chen-Gu/Zhao bounds \eqref{CGZ_bound}. Recall
    \begin{equation*}
        f_i(m) < u_i(m) < g_i(m)\,,\end{equation*}
       where $f_i(m) = \dfrac{(m-i)i}{(m - i+1)(i+1)}$  and $g_i(m) = \dfrac{(m-i)\,i\,(m+i^2 +1)}{(m-i+1) (i+1)(m+i^2)}.$\vskip2mm \noindent
       Notice that $g_i = f_i \left(1+ \dfrac{1}{m+i^2}\right)$. Since $u_i = \dfrac{r_{i+1}}{r_i}$, 
       $$f_i< \dfrac{r_{i+1}}{r_i} < f_i\left(1+ \dfrac{1}{m+i^2}\right).$$
       Telescoping (product) through $k=i,i+1,\dots,m-1$ gives
       \begin{equation}\label{simplified_CGZ}
           \dfrac{i}{m(m-i+1)} < \dfrac{r_m}{r_i} < \dfrac{i}{m(m-i+1)}  \prod_{k=i}^{m-1}\left(1+ \dfrac{1}{m+k^2}\right).
       \end{equation}
       It remains to find $r_m$. Applying the recurrence \eqref{BM_recurrence} with $i=m+1$ and simplifying gives
       $$r_m=\dfrac{d_m}{d_{m-1}} = \dfrac{2}{2m+1}.$$After substituting $r_m$ into \eqref{simplified_CGZ}
       \begin{equation*}
           \dfrac{2m(m-i+1)}{i(2m+1)} \cdot\dfrac{1}{\displaystyle\prod_{k=i}^{m-1}\left(1+ \dfrac{1}{m+k^2}\right)}< r_i <\dfrac{2m(m-i+1)}{i(2m+1)}. 
       \end{equation*}
       which proves the result.
       \end{proof}
       As immediate consequences, we obtain the following localization results needed for both interior and growing left regimes. 
       \begin{coro}[Localization] \label{Localization}
       \begin{enumerate}[label=\normalfont(\roman*)]\,\,\,
       \vskip2mm
       \item {\bf(Interior)} $r_i = R_0(\theta_i)+ O_\epsilon(h).$
 \item {\bf (Inner left-edge regime)} $s_i = 1+ O_C(m^{-1/2}).$ 
 \item {\bf (Outer left-edge regime)} $\dfrac{r_i- R_0(\theta_i)}{R_0(\theta_i)} = O(i^{-1}) = o(\theta_i).$ 
\end{enumerate}
\end{coro}

Let us note that these estimates are uniform over their corresponding regimes, with the implied constants depending only on the fixed regime parameters. In the outer-left case, the $o(\theta_i)$-term is uniform along ranges for which $i/\sqrt m\to\infty$ (equivalently, $h/\theta_i^2\to0$).

\begin{proof}\,\,\,
\underline{ \it Proof of (i).}\vskip3mm \noindent
              Recall $i/m \in I_\epsilon = [\epsilon/2 , 1- \epsilon/2]$ in the interior regime. First
           $$\log P_i = \sum_{k=i}^{m-1}\log\left(1+ \dfrac{1}{m+k^2}\right) \leq \sum_{k=i}^{m-1} \dfrac{1}{m+k^2}  \leq \sum_{k=i}^\infty \dfrac{1}{k^2} = O_\epsilon (i^{-1}) = O_\epsilon (i^{-1}) = O_\epsilon (h).$$
           Therefore $P_i = 1+ O_\epsilon(h)$, which implies $\dfrac{1}{P_i} =1+ O_\epsilon(h) $.\vskip2mm \noindent
           Now we compare $U_i$ with $R_0(\theta_i) = \dfrac{m-i}{i}$. Indeed
           $$U_i - R_0(\theta_i) =\dfrac{2m(m-i+1)}{i(2m+1)} - \dfrac{m-i}{i}  = \dfrac{m+i}{i(2m+1)}.$$
Since $i \geq \epsilon m/2$, $\dfrac{m+i}{i(2m+1)} = O_\epsilon(m^{-1}) = O_\epsilon(h)$. Hence $$U_i = R_o(\theta_i)+O_\epsilon(h).$$
On the other hand, since $U_i = O_\epsilon(1)$ on $I_\epsilon$, we have
$$\dfrac{U_i}{P_i} =  U_i(1+ O_\epsilon(h)) = U_i+ O_\epsilon(h) = R_0(\theta_i)+O_\epsilon(h).$$ Thus 
$$R_0(\theta_i)+O_\epsilon(h)<r_i< R_0(\theta_i)+O_\epsilon(h).$$ Consequently, $r_i = R_0(\theta_i)+ O_\epsilon(h)$.

\vskip10mm \noindent
\underline{ \it Proof of (ii).}\vskip3mm \noindent
           Recall in the inner left-regime, $i=O(\sqrt{m})$.\vskip2mm \noindent
        Multiplying \eqref{telescope_CGZ} by $\dfrac{i}{m}$ gives
        \begin{equation} \label{CGZ_si}
           \dfrac{2(m-i+1)}{2m+1} \cdot\dfrac{1}{\displaystyle\prod_{k=i}^{m-1}\left(1+ \dfrac{1}{m+k^2}\right)}< s_i <\dfrac{2(m-i+1)}{2m+1}. 
       \end{equation}       
       Let $C>0$ be fixed. Then for $i \leq C\sqrt{m}$,
       $$\dfrac{2(m-i+1)}{2m+1} = 1 - \dfrac{2i-1}{2m+1} = 1+ O_C(m^{-1/2}).$$
Also, $$\log P_i = \sum_{k=i}^{m-1}\log\left(1+ \dfrac{1}{m+k^2}\right) \leq \sum_{k=0}^{\infty}\left(\dfrac{1}{m+k^2}\right) \leq \dfrac{1}{m} + \int_0^\infty \dfrac{1}{m +x^2}\,dx = \dfrac{1}{m}+ \dfrac{\pi}{2\sqrt{m}}.$$
The last expression is $O_C(m^{-1/2})$, and hence
       $$P_i = e^{\log P_i} = e^{O(m^{-1/2})} = 1+ O(m^{-1/2}).$$
       Thus $\dfrac{1}{P_i} = 1+ O(m^{-1/2})$. \vskip2mm \noindent
       So both endpoints in \eqref{CGZ_si} are the same; therefore, $s_i = 1+ O_C(m^{-1/2})$ uniformly for $i \leq C\sqrt{m}$.

       \vskip10mm \noindent
        \underline{ \it Proof of (iii).}\vskip3mm \noindent
Recall in the outer left-edge regime, $i>>\sqrt{m},  \dfrac{i}{\sqrt{m}} \to \infty$ and $\dfrac{i}{m} \to 0$. \vskip2mm
\noindent
Since $i>>\sqrt{m}$
$$\log P_i \leq \sum_{k=i}^{m-1} \dfrac{1}{m+k^2} \leq \sum_{k=i}^\infty \dfrac{1}{k^2} = O(i^{-1}).$$
Thus $P_i = 1+ O(i^{-1})$.\vskip2mm \noindent
Similarly, $$ \dfrac{2(m-i+1)}{2m+1}  = 1-\dfrac{i}{m} + O(m^{-1}) = 1-\theta_i + O(h).$$
Since $h/i^{-1} = i/m \to 0 \implies h = o(i^{-1})$, this simplifies to $$ \dfrac{2(m-i+1)}{2m+1}  = 1-\dfrac{i}{m} + O(m^{-1}) = 1-\theta_i + O(i^{-1}).$$ As a result, we have
\begin{equation}
s_i = 1-\theta_i + O(i^{-1}).
\end{equation}\label{localization_outer_left}
       Now $r_i = (m/i)s_i = s_i/\theta_i$. So, (54) implies
       $$r_i = \dfrac{1- \theta_i}{\theta_i} + O(\frac{1}{i\theta_i}) = R_0(\theta_i)+ O(\frac{1}{i\theta_i}).$$
       Consequently, $$\frac{r_i}{R_0(\theta_i)} = 1+ O(i^{-1}).$$ 
       Equivalently
       $$\frac{r_i- R_0(\theta_i)}{R_0(\theta_i)} = O(i^{-1}) = o(\theta_i)\,\,\,\,\,(\text{because $i^{-1}/\theta = m/i^2 \to 0$}).$$ 
       
      \end{proof}

\begin{lemma}[Logarithmic expansion] \label{logarithmic expansion}
    $$\log\left(\dfrac{i^3 (i+2)}{(i-1)(i+1)^3}\right) = \dfrac{2}{i^3}+ O(i^{-4}).$$
\end{lemma}
\begin{proof}
    After factoring out $i^4$, 
    $$\dfrac{i^3 (i+2)}{(i-1)(i+1)^3} = \dfrac{1 + \frac{2}{i}}{(1 - \frac{1}{i})(1+ \frac{1}{i})^3}.$$
    Taking the logarithm on both sides:
    $$\log\left(\dfrac{i^3 (i+2)}{(i-1)(i+1)^3}\right)  = \log\left(1+ \frac{2}{i}\right) - \log\left(1- \frac{1}{i}\right) - 3\log\left(1+ \frac{1}{i}\right).$$
    Using $\log(1+x) = x-\dfrac{x^2}{2}+ \dfrac{x^3}{3}-\dfrac{x^4}{4}+ O(x^5)$, we can expand out each logarithmic term on the RHS
    \begin{align*}
        \log\left(1+ \frac{2}{i}\right) & = \dfrac{2}{i} - \dfrac{2}{i^2}+ \dfrac{8}{3i^3}-  \dfrac{4}{i^4} + O(i^{-5}).\\
        -\log\left(1- \frac{1}{i}\right) & = \dfrac{1}{i} + \dfrac{1}{2i^2}+ \dfrac{1}{3i^3} +  \dfrac{1}{4i^4} + O(i^{-5}).\\
        -3\log\left(1+ \frac{1}{i}\right) & = -\dfrac{3}{i} + \dfrac{3}{2i^2}- \dfrac{1}{i^3} +  \dfrac{3}{4i^4} + O(i^{-5}).
    \end{align*}
    Adding the three terms result in cancellations of $i^{-1}$ and $i^{-2}$ terms. Therefore,
    $$\log\left(\dfrac{i^3 (i+2)}{(i-1)(i+1)^3}\right) = \dfrac{2}{i^3} - \dfrac{3}{i^4} + O(i^{-5}).$$
    This proves the result since $O(i^{-4}) + O(i^{-5}) = O(i^{-4})$ as $i \to \infty$.
\end{proof}

\newpage
\begin{lemma} [Third difference estimates in the outer left region]\label{third difference} For $\dfrac{i}{\sqrt{m}} \to \infty$ and $i=o(m)$, 

\begin{align*}
    -\Delta^3 \log R_0(\theta_{i-1}) &= \dfrac{2}{i^3} + o(i^{-3}).\\
    -\Delta^3 \log \left(1 -\dfrac{1}{2(i-1)}+ \dfrac{1}{m-(i-1)}\right) &= o(i^{-3}).\\
   - \Delta^3 \log(1+ u_{i-1})&= o(i^{-3}).
    \end{align*}
\end{lemma}
\begin{proof}
For the first term, recall $R_0(\theta_i) = \dfrac{m-i}{i}$.  By taking logarithm on both sides
$$\log R_0(\theta_i)  = \log(m-i) - \log i$$
which, after taking the third discrete difference, equals
$$-\Delta^3 \log R_0(\theta_{i-1}) =  \Delta^3 \log (i-1) -\Delta^3 \log (m-i+1).$$
By definition
$$\Delta^3 \log (i-1) = \log(i+2) - 3 \log(i+1) +3 \log i - \log (i-1).$$
which is the same as $\log \left(\dfrac{i^3(i+2)}{(i-1)(i+1)^3}\right)$. Then by \eqref{logarithmic expansion}, we have
$$\Delta^3 \log (i-1) = \dfrac{2}{i^3}+ O(i^{-4}).$$
Similarly, $$-\Delta^3 \log (m-i+1) = \dfrac{2}{(m-i)^3}+ O((m-i)^{-4}).$$
Combining the two estimates give
$$-\Delta^3 \log R_0(\theta_{i-1}) = \dfrac{2}{i^3} +\dfrac{2}{(m-i)^3}+ O(i^{-4}) + O((m-i)^{-4}).$$
Since $i= o(m)$,  $\left(\dfrac{(m-i)^{-3}}{i^{-3}}\right) = \left(\dfrac{i}{m-i}\right)^3 \to 0$. Therefore,
$$-\Delta^3 \log R_0(\theta_{i-1}) =\dfrac{2}{i^3} + o(i^{-3}).$$
For the second term, start with the factorization
$$1 - \dfrac{1}{2x}+ \dfrac{1}{m-x} = \left(1 - \dfrac{1}{2x}\right) \left(1+ \dfrac{2x}{(2x-1)(m-x)}\right).$$
Let $F = \log\left(1 - \dfrac{1}{2x}\right)$ and $G= \log\left(1+ \dfrac{2x}{(2x-1)(m-x)}\right)$. \vskip2mm \noindent
Direct computation gives
$$F'''(x) = O(x^{-4}).$$ and then by Taylor's third difference formula, we have
$$\Delta^3 F(i-1) = O(i^{-4}) = o(i^{-3})\,\,\,\,\,\,\,\,\,\text{(note that $F^{(4)}(x) = O(x^{-5})$)}.$$
For the second, put $b(i) \coloneq \dfrac{2i}{(2i-1)(m-i)}$ and notice $\dfrac{2i}{(2i-1)(m-i)} = \dfrac{1}{m-i} \left(1 + \dfrac{1}{2i-1}\right)$. Then, since $i = o(m)$, $b(i) = O(m^{-1}) = o(1)$. Moreover, regarding $b(x)$ as a smooth function near $x = i+ O(1)$, differentiating three times gives terms of the form
$$O((m-i)^{-4}), O(i^{-1}(m-i)^{-3}), O(i^{-2}(m-i)^{-2})\,\,\,\text{and}\,\,\, O(i^{-3}(m-i)^{-1})$$
But because $i = o(m)$, each one of these terms is $o(i^{-3})$. And, $G^{(4)}(x) = o(i^{-3})$. Moreover since $b(i) = o(1)$, the same holds for the third derivative of $\log(1+ b(x))$. Hence by Taylor's formula
$$\Delta^3 G(i-1) = o(i^{-3}).$$
Combining the two estimates give
$$-\Delta^3 \log \left(1 - \dfrac{1}{2(i-1)}+ \dfrac{1}{m-(i-1)} \right)= o(i^{-3}).$$
Finally, for the third term, after substituting for $\hat{r}_i$ 
$$u_i = \dfrac{E_i}{\hat{r}_i} = \dfrac{E_i}{R_o(\theta_i)\left(1 - \dfrac{1}{2i}+ \dfrac{1}{m-i} \right)}.$$
Since $R_0(\theta_i) = \dfrac{1-\theta_i}{\theta_i}$
$$\dfrac{1}{\hat{r}_i} = \dfrac{\theta_i}{1-\theta_i}\left(1 - \dfrac{1}{2i}+ \dfrac{1}{m-i} \right)^{-1}.$$
In the outer left regime, the second term is  $1 + O(i^{-1})$, thus
$$\dfrac{1}{\hat{r}_i} =  \theta_1 (1+ O(\theta_i)+ O(i^{-1})).$$
This implies
$$u_i \sim \theta_i E_i.$$
Now, because the estimate is $\Delta^k E_i = O(h^{k+2}\theta_i^{-k-4})$ for $k=0,1,2,3$, multiplying by $\theta_i$ gives
$$\Delta^k u_i = O(h^{k+2}\theta_i^{-k-3})\,,\,\,\,k=0,1,2,3.$$
Indeed for $k=1,2,3$, this follows from the discrete product rule applied to the explicit expression for $1/\hat{r}_i$ ; the difference of the $O(i^{-1})+O(\theta_i)$ correction are of lower order in the outer-left regime.
\vskip2mm \noindent
To pass to $\log(1+u_i)$, since $u_i = o(1)$, the derivatives of $x\mapsto\log(1+x)$ through order three are uniformly bounded on the relevant range. Repeated finite-difference identities (or the discrete chain/product rule) therefore yield
$$\Delta^3 \log(1+u_i) = O(\Delta^3 u_i) + O(\Delta u_i \Delta^2u_i)+ O((\Delta u_i)^3).$$
The leading term $O(\Delta^3 u_i) = O(h^5 \theta_i^{-6}) = o(i^{-3})$ since $h^5 \theta_i^{-6}/i^{-3} = m/i^3 \to 0$ in the outer left regime, i.e. $i/\sqrt{m} \to \infty$ and $i = o(m)$. Therefore,
$$\Delta^3 \log(1+u_i) = o(i^{-3}).$$

\end{proof}

\vskip2cm

%\pagebreak

\noindent Heshan Aravinda \\
Department of Mathematics and Statistics \\
University of Wyoming \\
1000 E. University Ave\\
Laramie, WY 82071\\
Email: hpathira@uwyo.edu

\end{document}